\documentclass[final]{amsart}
\usepackage{amssymb}
\usepackage{hhline}
\usepackage{graphicx}
\usepackage{stmaryrd}
\usepackage{float}
\usepackage{cite}

\usepackage[utf8]{inputenc}
\usepackage{listings}
\usepackage{xcolor}

\newtheorem{theorem}{Theorem}[section]

\newtheorem{proposition}[theorem]{Proposition}
\newtheorem{corollary}[theorem]{Corollary}

\newtheorem{problem}[theorem]{Problem}
\theoremstyle{definition}

\newtheorem{remark}[theorem]{Remark}

\newcommand{\Aut}{\mathrm{Aut\mkern 2mu}}
\newcommand{\Halo}{\mathrm{Halo\mkern 2mu}}

\title{On the number of $g$-dimonoids of small order}

\author{Volodymyr M. Gavrylkiv}
\address[V.~Gavrylkiv]{Vasyl Stefanyk Carpathian National University, Ivano-Frankivsk, Ukraine} \email{vgavrylkiv@gmail.com}

\begin{document}

\begin{abstract}
We study $g$-dimonoids, that is, algebraic structures endowed with two associative binary operations satisfying a specified system of axioms. The paper investigates duality and isomor\-phisms of $g$-dimonoids and provides a complete characterization of all rectangular commuta\-tive $g$-dimonoids. Several new examples are constructed, including commutative iso-dual $g$-dimonoids that are not abelian, and noncommutative nonabelian rectangular $g$-dimonoids that are not dimonoids. A complete classification of all two-element $g$-dimonoids is obtained. Finally, we present the results of computer computations determining the numbers of all pairwise non\-isomorphic $g$-dimonoids of orders up to~$5$, as well as all pairwise nonisomorphic commutative, abelian, and rectangular $g$-dimonoids of orders up to~$6$, obtained using \texttt{GAP}, \texttt{Python}, and \texttt{C++}.

\end{abstract}

\maketitle

\section*{Introduction}

The concept of a dialgebra was introduced by Jean-Louis Loday~\cite{Lod} in his search for a class of (linear) algebras that generate Leibniz algebras in a manner analogous to how associative algebras give rise to Lie algebras. Recall that a Leibniz algebra is a linear algebra over a field whose bracket operation $[, ]$ satisfies the Leibniz derivation identity
$$[[x, y], z] = [[x, z], y] + [x, [y, z]],$$
without necessarily being anticommutative.

Loday’s idea was to “separate” the left and right multiplications, treating them as distinct associative operations, denoted $\vdash$ and $\dashv$.
He observed that if these operations satisfy the following three axioms:
\begin{align*}
(x \dashv y) \dashv z &= x \dashv (y \vdash z), \hspace{10mm}(D_1) \\
(x \vdash y) \dashv z &= x \vdash (y \dashv z), \hspace{10mm}(D_2)\\\
(x \dashv y) \vdash z &= x \vdash (y \vdash z), \hspace{10mm}(D_3)
\end{align*}
then the bracket defined by $[x, y] = x \dashv y - y \vdash x$ satisfies the Leibniz identity.

Accordingly, Loday defined a {\em dimonoid}~\cite{Lod} as an algebraic structure $(D,\dashv,\vdash)$ consisting of a set $D$ equipped with two associative binary operations $\dashv$ and $\vdash$ satisfying axioms $(D_1)$, $(D_2)$, and $(D_3)$.  Since dialgebras are the linear analogues of dimonoids, many results on dimonoids have direct applications in the theory of dialgebras~\cite{B,F,Lod,M,ZCY}. 
T.~Pirashvili~\cite{P} introduced the concept of a duplex, a generalization of dimonoids, and constructed the free duplex.
The properties of free dimonoids were employed in~\cite{Lod} to characterize free dialgebras and to study their cohomologies. In~\cite{Liu}, the notion of a dimonoid was used to define and investigate one-sided dirings. Furthermore, dimonoids are closely related to restrictive bisemigroups~\cite{Sh} and doppelsemigroups~\cite{GR1, GDS2, Zh2017AU}.

In~\cite{BG9, G7, G9, GR1, GDS2, GUDS, Gdim1, Gdim2, GSDS, GLDS, GS_sie, Gdim3, GDSso}, several classes of dimonoids, doppelsemigroups, and extensions of doppelsemigroups were constructed and studied. Among numerous results, these works provide isomor\-phism criteria for the corresponding classes, structural classifications of objects of small order, descriptions of their automorphism groups, and enumerations of pairwise nonisomor\-phic structures of small order.
 
In~\cite{MDS}, the notion of a {\em generalized dimonoid} (or {\em $g$-dimonoid}) was introduced as an algebraic structure $(D, \dashv, \vdash)$ consisting of a set $D$ endowed with two associative binary operations $\dashv$ and $\vdash$ that satisfy the axioms $(D_1)$ and $(D_3)$. The construction of the free $g$-dimonoid was also described therein. Every semigroup $(D, \dashv)$ can naturally be regarded as a $g$-dimonoid $(D, \dashv, \dashv)$, referred to as the {\em trivial $g$-dimonoid}. In this sense, $g$-dimonoids constitute a genuine generalization of semigroups. 

An {\em associative $0$-dialgebra}~\cite{Po}, that is, a linear space over a field equipped with two associative binary operations $\dashv$ and $\vdash$ satisfying the same axioms $(D_1)$ and $(D_3)$, can be viewed as a linear analogue of a $g$-dimonoid. Given the central role of dimonoids and $g$-dimonoids in the study of Leibniz algebras, dialgebras, and associative 0-dialgebras, their investigation by semigroup-theoretic methods constitutes a natural and promising line of research.

In~\cite{ZhYulia1}, a free $n$-nilpotent $g$-dimonoid was constructed, the least $n$-nilpotent congruence on a free $g$-dimonoid was examined, and a characterization of a free $g$-dimonoid was obtained. The construction of the free commutative $g$-dimonoid together with the description of the least commutative congruence on a free $g$-dimonoid was provided in~\cite{ZhYulia2}. Furthermore, Yu.~Zhuchok~\cite{ZhYurii_gdm2016} determined all isomorphisms between the endomorphism semigroups of free commutative $g$-dimonoids and proved that every automorphism of such an endomorphism semigroup is quasi-inner. In~\cite{MY}, Cayley-type theorems for $g$-dimonoids were proved using left and right actions of sets and the concept of a dialgebra. More recently, A.~Zhuchok~\cite{ZhAEJM2022} described the least dimonoid congruence and the least semigroup congruence on the free (commutative, $n$-nilpotent) $g$-dimonoid. In~\cite{Ggdim2}, a  construction of $g$-dimonoids based on inflations of semigroups was introduced. In addition, several classes of $2$-dinilpotent commutative $g$-dimonoids were established, isomorphism criteria for these classes were obtained, and the dimonoids among them were characterized.

The present paper investigates duality and isomorphisms of $g$-dimonoids. We introduce several new classes of $g$-dimonoids and determine their automorphism groups and halos. Based on these constructions, a complete classification, up to isomorphism, of all two-element $g$-dimonoids is obtained. Moreover, all rectangular commutative $g$-dimonoids are fully characterized. Furthermore, we construct examples of commutative iso-dual $g$-dimo\-noids that are nonabelian, as well as noncommutative nonabelian rectangular $g$-dimonoids that are not dimonoids. Finally, computer-assisted calculations yield the numbers of all pairwise nonisomorphic $g$-dimonoids of orders up to~$5$, and all pairwise nonisomorphic commutative, abelian, and rectangular $g$-dimonoids of orders up to~$6$.

\section{Preliminaries on semigroups}

An element $e$ of a semigroup $(S,*)$ is called  {\em a left identity} (resp.  {\em a right identity}) in $S$ if $e*a=a$ (resp. $a*e=a$) for any $a\in S$. An element 1 is an \emph{identity} if it is both a left and a right identity. An element $e$ of a semigroup $(S,*)$ is called  {\em a middle identity} in $S$ if $\ell*e*r=\ell*r$  for all $\ell,r\in S$. It is clear that each left (right) identity is a middle identity.

An element $e$ of a semigroup $(S,*)$ is called an {\em idempotent} if $e*e=e$.  The semigroup is a {\em band}, if all its elements are idempotents. Commutative bands are called {\em semilattices}.
By $L_n$ we denote the {\em linear semilattice} $\{0, 1,\ldots, n\!-\!1\}$ of order $n$, endowed with the operation of minimum.

\smallskip

An element $z$ of a semigroup $S$ is called  {\em a left zero} (resp.  {\em a right zero}) in $S$ if $z*a=z$ (resp. $a*z=z$) for any $a\in S$. An element 0 is called a \emph{zero} if it is both a left  and a right zero.

Let $(S,*)$ be a semigroup and $0\notin S$. The binary operation $*$ defined on  $S$  can be extended to $S\cup\{0\}$ putting $0*s=s*0=0$ for all $s\in S\cup \{0\}$.
The notation $(S,*)^{+0}$  denotes a semigroup $(S\cup\{0\},*)$ obtained from $(S,*)$ by adjoining the extra zero $0$ (regardless of whether $(S,*)$ has a zero). 

\smallskip

A semigroup $(S,*)$ is called a {\em null semigroup} if there exists an element $0\in S$ such that $x*y=0$ for all $x,y\in S$. In this case  $0$ is a zero of $S$.  By $O_{S^0}$ we denote a null semigroup with zero $0$ on a set $S$.  The null semigroups $O_{S^0}$ and $O_{T^z}$ are isomorphic if and only if $|S|=|T|$. If $S$ is a  set of cardinality $|S|=n$, we use the notation $O_n$ for a representative of the class of semigroups isomorphic to $O_{S^0}$.

\smallskip

If $(S,*)$ is a semigroup, then the semigroup $(S,{*}^d)$ with operation $x{*}^d y=y* x$ is called {\em dual} to $(S,*)$, denoted $(S,*)^d$. It follows that $(S,*)^d = (S,*)$ if and only if $(S,*)$ is a commutative semigroup, and $\Aut(S,*)^d=\Aut(S,*)$.

\smallskip

A semigroup $(S,*)$ is said to be a {\em left} (resp. {\em right}) {\em zero semigroup} if $a*b=a$ (resp. $a*b=b$) for any $a,b\in S$. By  $LO_S$ and $RO_S$ we denote  a left zero  and a right zero semigroup on a set $S$, respectively.   If $S$ is a set of cardinality $|S|=n$,  we use the notations $LO_n$ and $RO_n$ for representatives of the classes of semigroups isomorphic to $LO_S$ and $RO_S$, respectively.

\smallskip

A semigroup $(S, *)$ is called \emph{rectangular}~\cite{ZhCA2017a, GS_sie} if every element of $S$ is a middle identity, that is,
$x*y*z = x*z$ for all  $x, y, z \in S$. In other words, the product in $S$ depends only on the first and the last factors. 
A nontrivial null semigroup is an  example of a commutative rectangular semigroup that has no identity element. The semigroups $LO_S$ and $RO_S$ are dual rectangular bands.
In~\cite{Nagy}, it was proved that a semigroup $(S,*)$ is rectangular if and only if the factor semigroup $S/\theta$ is a right zero semigroup, 
where $\theta = \{(a, b)\in S\times S : x*a = x*b \text{ for all } x\in S\}$. The authors also described a method for constructing all rectangular semigroups. In~\cite{GS_sie}, the study focused on rectangular semigroups, including rectangular ideal extensions of left (right) zero semigroups by null quotients, and provided new combinatorial results for the numbers of pairwise nonisomorphic instances.

\smallskip

Let $S$ be a nonempty set, $A \subseteq S$ a nonempty subset, and $a \in A$.
Define an associative binary operation $*$ on $S$ as follows:

$$x* y=\begin{cases}
x\quad \text{if } x\in A \\
a\quad \text{if } x \notin A.
\end{cases}$$

We denote the semigroup $(S,*)$  by $LO_{A_a\leftarrow S}$. It follows that  all elements $z\in A$ are left zeros of $LO_{A_a\leftarrow S}$. If $A=\{a\}$, then $LO_{A_a\leftarrow S}$ coincides with a null semigroup $O_{S^a}$ with zero $a$. If $A=S$, then $LO_{A_a\leftarrow S}$ coincides with a left zero semigroup $LO_S$. The semigroups $LO_{A_a\leftarrow S}$ and $LO_{B_b\leftarrow T}$  are isomorphic if and only if $|S|=|T|$ and $|A|=|B|$. If $S$ is a finite set of cardinality $|S|=n$ and $|A|=m$,  we use the notation $LO_{m\leftarrow n}$ for a representative of the class of semigroups isomorphic to  $LO_{A_a\leftarrow S}$.

By $RO_{A_a\leftarrow S}$ we denote a dual semigroup of $LO_{A_a\leftarrow S}$ and use the notation $RO_{m\leftarrow n}$ accordingly. Both $LO_{A_a\leftarrow S}$ and $RO_{A_a\leftarrow S}$ are rectangular semigroups. Each of them is a rectangular band precisely when $A=S$.

\bigskip

Following the algebraic tradition, we take for a model of the class of cyclic groups of order $n$ the multiplicative group
$C_n=\{z\in\mathbb C:z^n=1\}$ of $n$-th roots of $1$. For a set $X$,  we denote by $S_X$ the group of all bijections of $X$.

\section{Some definitions and basic properties of $g$-dimonoids}

In this section, we recall several useful results on dimonoids  that will be  used in the subsequent investigations.

An element $e$ of a $g$-dimonoid $(D,\dashv, \vdash)$ is called a {\em bar-unit} if $e \vdash d = d = d \dashv e$ for all $d\in D$. In contrast to monoids a $g$-dimonoid may have many bar-units. The set of all bar-units of a $g$-dimonoid $(D,\dashv, \vdash)$ is called the {\em halo} of $(D,\dashv, \vdash)$, denoted $\Halo(D,\dashv, \vdash)$. A nonempty subset $B\subset D$ is called a {\em $g$-subdimonoid}~\cite{ZhYulia1} of a $g$-dimonoid $(D,\dashv, \vdash)$  if $a \dashv b, a \vdash b \in B$ for all $a, b \in B$. If the halo of $(D,\dashv, \vdash)$ is nonempty, then it is a $g$-subdimonoid of $(D,\dashv, \vdash)$.

An element $0\in D$  is called  a {\em   zero of a $g$-dimonoid} $(D,\dashv, \vdash)$~\cite{ZhYulia1} if $0$ is a zero of $(D,\dashv)$ and a  zero of $(D,\vdash)$. Let $(D,\dashv, \vdash)$ be a $g$-dimonoid and $0\notin D$. The binary operations defined on  $D$  can be extended to $D\cup\{0\}$ putting $0\dashv d=d\dashv 0=0=0\vdash d=d\vdash 0 $ for all $d\in D\cup \{0\}$. The notation $(D,\dashv, \vdash)^{+0}$  denotes a $g$-dimonoid $D\cup\{0\}$ obtained from $D$ by adjoining the extra zero $0$. 
It follows  that $\Halo((D,\dashv, \vdash)^{+0}) = \Halo(D,\dashv, \vdash)$.

\smallskip

The axioms of a $g$-dimonoid imply the following proposition.

\begin{proposition}\label{grl_id} Let $(D,\dashv, \vdash)$ be a $g$-dimonoid. If the semigroup $(D,\dashv)$ contains a left identity or the semigroup $(D,\vdash)$ contains a right identity, then the operations of a $g$-dimonoid $(D,\dashv, \vdash)$ coincide.
\end{proposition}

\smallskip

Let $(D,\dashv, \vdash)$ be a $g$-dimonoid. Define new operations $\dashv^d$ and  $\vdash^d$ on $D$ by 
$$x \dashv^d y = y \vdash x\ \ \text{  and  }\ \  x \vdash^d y = y \dashv x.$$
It is immediate to check that $(D,\dashv^d, \vdash^d)$ is a new $g$-dimonoid, called the {\em  dual $g$-dimonoid of $(D,\dashv, \vdash)$}, which we denote by $(D,\dashv, \vdash)^d$. It follows that the unary duality operation $d: (D, \dashv, \vdash) \mapsto (D, \dashv, \vdash)^{d}$ is involutive in the sense that $((D,\dashv, \vdash)^d)^d=(D,\dashv, \vdash)$. In fact, $(D,\dashv, \vdash)^d$ is a $g$-dimonoid if and only if $(D,\dashv, \vdash)$ is a $g$-dimonoid.  As usual, a $g$-dimonoid $(D,\dashv, \vdash)$ is said to be {\em self-dual} if $(D,\dashv, \vdash)^d=(D,\dashv, \vdash)$. 

\begin{remark}Observe that if we put
$x \dashv' y = y \dashv x$  and  $x \vdash' y = y \vdash x$, then, in general, the structure $(D,\dashv', \vdash')$ is not a $g$-dimonoid. Indeed, consider the left-zero and right-zero dimonoid (and hence a $g$-dimonoid) $(D, \dashv, \vdash)$, where the operations are given by $x\dashv y = x$ and $x\vdash y=y$ (see~\cite{Lod}). Taking into account that $x \dashv' y = y \dashv x = y$ and $x \vdash' y = y \vdash x = x$, we conclude that $ (x \dashv' y) \dashv' z = z$ while   $x \dashv' (y \vdash' z) = x \dashv' y = y$  for all $x,y,z\in D$. Hence, the axiom $(D_1)$ fails, and consequently $(D,\dashv', \vdash')$ can not be a $g$-dimonoid.
\end{remark}

\smallskip

A $g$-dimonoid $(D,\dashv, \vdash)$ is called {\em abelian} if $x \dashv y = y \vdash x$ for all $x,y\in D$.

\begin{proposition}\label{gda}
Let $(D,\dashv, \vdash)$ be a $g$-dimonoid. Then the following conditions are equivalent:
\begin{itemize}
\item[1)] $(D,\dashv)$ and $(D,\vdash)$ are dual semigroups;
\item[2)] $(D,\dashv, \vdash)$ is abelian;
\item[3)] $(D,\dashv, \vdash)$ is self-dual.
\end{itemize}
\end{proposition}

\begin{proof}[Proof] $(1)\Rightarrow(2)$ If $(D,\dashv)$ and $(D,\vdash)$ are dual semigroups, then $x\dashv y = y\vdash x $ for all $x,y\in D$, and hence $(D,\dashv, \vdash)$ is an abelian $g$-dimonoid.

$(2)\Rightarrow(3)$ Let $(D,\dashv, \vdash)$ be an abelian $g$-dimonoid. Taking into account that 
$$x \dashv^d y = y \vdash x = x \dashv y\ \ \text{  and  }\ \  x \vdash^d y = y \dashv x = x \vdash y,$$
we conclude that $\dashv^d\ =\ \dashv$ and $\vdash^d\ =\ \vdash$, and hence $(D,\dashv, \vdash)$ is a self-dual $g$-dimonoid.

$(3)\Rightarrow(1)$ If $(D,\dashv, \vdash)$ is a self-dual $g$-dimonoid, then $x \dashv y = x \dashv^d y = y \vdash x$ for all $x,y\in D$, and hence $(D,\dashv)$ and $(D,\vdash)$ are dual semigroups.
\end{proof}

\begin{corollary} The class of nonabelian $g$-dimonoids decomposes into pairs of dual $g$-dimonoids. 
\end{corollary}

The definition of the halo of a $g$-dimonoid  and Proposition~\ref{gda} imply the following corollary.

\begin{corollary}\label{ghaloabdim} Let $(D,\dashv, \vdash)$ be an abelian $g$-dimonoid. If the halo of $(D,\dashv, \vdash)$ is nonempty, then it coincides with the subsemigroup of all right identities of a semigroup $(D,\dashv)$, and the subsemigroup of all left identities of a semigroup $(D,\vdash)$.
\end{corollary}

\smallskip

A $g$-dimonoid $(D,\dashv,\vdash)$ is called {\em commutative}~\cite{ZhYulia2} if both semigroups $(D,\dashv)$ and $(D,\vdash)$ are commutative.

Since commutative semigroups $(D,\dashv)$ and $(D,\vdash)$ are dual if and only if their operations coincide, Proposition~\ref{gda} implies the following corollaries.

\begin{corollary}\label{com_na_gdm} Commutative nontrivial $g$-dimonoids are nonabelian.
\end{corollary}

\begin{corollary}\label{ab_ncom_gdm} Abelian nontrivial $g$-dimonoids are noncommutative.
\end{corollary}

A $g$-dimonoid $(D,\dashv, \vdash)$ is said to be {\em rectangular} if both $(D,\dashv)$ and $(D,\vdash)$ are rectangular semigroups.

The left-zero and right-zero dimonoid is an example of a noncommuative  abelian  rectan\-gular $g$-dimonoid. Theorem~\ref{iso-dual_O} gives examples of nonabelian  commuative rectangular $g$-dimo\-noids that are not dimonoids.

\smallskip

Observing that a semigroup $(S,*)$ is rectangular if and only if its dual is rectangular, we obtain the following proposition.

\begin{proposition}\label{ds_un_rec}
Let $(D,\dashv, \vdash)$ be a $g$-dimonoid. Then  $(D,\dashv, \vdash)$ is rectangular if and only if $(D,\dashv, \vdash)^d$ is rectangular.
\end{proposition}

\smallskip

For a $g$-dimonoid $(D,\dashv,\vdash)$, if $\mathbb S$ and $\mathbb T$ denote the semigroups $(D,\dashv)$ and $(D,\vdash)$, respectively, then $\mathbb S \rbag \mathbb T$ stands for the $g$-dimonoid $(D,\dashv,\vdash)$.

\section{Isomorphisms of $g$-dimonoids}

A map $\varphi : D_1 \to D_2$ is called a {\em homomorphism } from a $g$-dimonoid $(D_1,\dashv_1, \vdash_1)$ to a $g$-dimonoid $(D_2,\dashv_2, \vdash_2)$~\cite{ZhYurii_gdm2016} if $$\varphi(a\dashv_1 b)=\varphi(a)\dashv_2\varphi(b)\ \ \text{  and  }\ \ \varphi(a\vdash_1 b)=\varphi(a)\vdash_2\varphi(b)$$ for all $a,b\in D_1$.

A bijective homomorphism $\psi : D_1 \to D_2$ is called an {\em isomorphism } from a $g$-dimonoid $(D_1,\dashv_1, \vdash_1)$ to a $g$-dimonoid $(D_2,\dashv_2, \vdash_2)$.
If there exists an isomorphism from a $g$-dimonoid $(D_1,\dashv_1, \vdash_1)$ to a $g$-dimonoid $(D_2,\dashv_2, \vdash_2)$, then $(D_1, \dashv_1, \vdash_1)$ and $(D_2, \dashv_2, \vdash_2)$ are said to be {\em isomorphic}, denoted $(D_1,\dashv_1, \vdash_1)\cong (D_2,\dashv_2, \vdash_2)$. An isomorphism $\psi: D\to D$ is called an {\em   automorphism} of a $g$-dimonoid $(D,\dashv, \vdash)$. By $\Aut(D,\dashv, \vdash)$ we denote the automorphism group of a $g$-dimonoid $(D,\dashv, \vdash)$. It follows that $\Aut((D,\dashv, \vdash)^{+0}) \cong \Aut(D,\dashv, \vdash)$.

\begin{proposition}\label{homo_gdm_dual} Let $(D_1,\dashv_1, \vdash_1)$ and  $(D_2,\dashv_2, \vdash_2)$ be $g$-dimonoids. Then for an arbitrary map $\varphi: D_1 \to D_2$ the following conditions are equivalent:
\begin{itemize}
\item[1)] $\varphi$ is a homomorphism from $(D_1,\dashv_1, \vdash_1)$ to $(D_2,\dashv_2, \vdash_2)$;
\item[2)] $\varphi$ is a homomorphism from $(D_1,\dashv_1, \vdash_1)^{d}$ to $(D_2,\dashv_2, \vdash_2)^{d}$.
\end{itemize}
\end{proposition}

\begin{proof}[Proof]  

$(1)\Rightarrow(2)$ Let $\varphi$ be a homomorphism from $(D_1,\dashv_1, \vdash_1)$ to $(D_2,\dashv_2, \vdash_2)$.
Observing that
\[
\begin{aligned}
&\varphi(x\dashv_1^d y) = \varphi(y\vdash_1 x) = \varphi(y)\vdash_2 \varphi(x) = \varphi(x) \dashv_2^d \varphi(y) \text{ and} \\
&\varphi(x\vdash_1^d y) = \varphi(y\dashv_1 x) = \varphi(y)\dashv_2 \varphi(x) = \varphi(x) \vdash_2^d \varphi(y),
\end{aligned}
\]
for all $x,y\in D_1$, we deduce that $\varphi$ is a homomorphism from $(D_1,\dashv_1, \vdash_1)^{d}$ to $(D_2,\dashv_2, \vdash_2)^{d}$.

\medskip
$(2)\Rightarrow(1)$ Let $\varphi$ be a homomorphism from $(D_1,\dashv_1, \vdash_1)^{d}$ to $(D_2,\dashv_2, \vdash_2)^{d}$.
Noting that
\[
\begin{aligned}
&\varphi(x\dashv_1 y) = \varphi(y\vdash_1^d x) = \varphi(y)\vdash_2^d \varphi(x) = \varphi(x) \dashv_2 \varphi(y) \text{ and} \\
&\varphi(x\vdash_1 y) = \varphi(y\dashv_1^d x) = \varphi(y)\dashv_2^d \varphi(x) = \varphi(x) \vdash_2 \varphi(y),
\end{aligned}
\]
for all $x,y\in D_1$, we conclude that $\varphi$ is a homomorphism from $(D_1,\dashv_1, \vdash_1)$ to $(D_2,\dashv_2, \vdash_2)$.
\end{proof}

\begin{corollary}\label{iso_gdm_dual} Let $(D_1,\dashv_1, \vdash_1)$ and  $(D_2,\dashv_2, \vdash_2)$ be $g$-dimonoids. Then   $(D_1,\dashv_1, \vdash_1)$ and  $(D_2,\dashv_2, \vdash_2)$ are isomorphic if and only if $(D_1,\dashv_1, \vdash_1)^{d}$ and  $(D_2,\dashv_2, \vdash_2)^{d}$ are isomorphic.
\end{corollary}

\begin{corollary}\label{aut_gdm_dual} Let $(D,\dashv, \vdash)$ be a $g$-dimonoid. Then 
$  \Aut(D,\dashv, \vdash)^{d}=\Aut(D,\dashv, \vdash)$. 
\end{corollary}

\medskip

\begin{proposition}\label{gduality} Let $(D,\dashv, \vdash)$ be a $g$-dimonoid. Then $\Halo(D,\dashv, \vdash)^d = \Halo(D,\dashv, \vdash)$. More\-over, $(D,\dashv, \vdash)$ is commutative if and only if $(D,\dashv, \vdash)^d$ is commutative.
\end{proposition}

\begin{proof}[Proof] 
Let $e$ be a bar-unit of the $g$-dimonoid $(D,\dashv, \vdash)$. Then $e\vdash d = d = d\dashv e$ for all $d\in D$. Since $e\vdash^d d = d\dashv e = d = e\vdash d = d\dashv^d e$ for all $d\in D$, it follows that $e$ is a bar-unit of the $g$-dimonoid $(D,\dashv, \vdash)^d$ as well. Using involutivity of the unary duality operation, we conclude that $\Halo(D,\dashv, \vdash)^d = \Halo(D,\dashv, \vdash)$.

Since commutativity of an operation $\dashv$ is equivalent to commutativity of an operation $\vdash^d$ and commutativity of an operation $\vdash$ is equivalent to commutativity of an operation $\dashv^d$, we conclude that $(D,\dashv, \vdash)$ is commutative if and only if $(D,\dashv, \vdash)^d$ is commutative.
\end{proof}

\begin{proposition}\label{gisoLO} Let $(D_1,\dashv_1, \vdash_1)$ and $(D_2,\dashv_2, \vdash_2)$ be  $g$-dimonoids such that  and $(D_1,\dashv_1)$ and $(D_2,\dashv_2)$ are left zero semigroups. The $g$-dimonoids $(D_1,\dashv_1, \vdash_1)$ and $(D_2,\dashv_2, \vdash_2)$ are isomorphic if and only if the semigroups $(D_1,\vdash_1)$ and $(D_2,\vdash_2)$ are isomorphic.
\end{proposition}

\begin{proof}[Proof] It is immediate to observe that if the $g$-dimonoids $(D_1,\dashv_1, \vdash_1)$ and $(D_2,\dashv_2, \vdash_2)$ are isomorphic, then the semigroups $(D_1,\vdash_1)$ and $(D_2,\vdash_2)$ are also isomorphic.
Conversely, let $\psi: D_1 \to D_2$ be an isomorphism from the semigroup $(D_1,\vdash_1)$ to the semigroup $(D_2,\vdash_2)$. Then, necessarily, $|D_1| = |D_2|$. Taking into account that any bijective map is an isomorphism from the left zero semigroup $(D_1,\dashv_1)$ to the left zero semigroup $(D_2,\dashv_2)$, it follows that $\psi$ is also an isomorphism from the left zero semigroup $(D_1,\dashv_1)$ to the left zero semigroup $(D_2,\dashv_2)$. Therefore, $\psi$ is an isomorphism from the $g$-dimonoid $(D_1,\dashv_1, \vdash_1)$ to the $g$-dimonoid $(D_2,\dashv_2, \vdash_2)$.
\end{proof}

\begin{corollary}\label{gaut_LO_dm} Let $(D,\dashv, \vdash)$ be a $g$-dimonoid such that  and $(D,\dashv)$ is a left zero semigroup. Then $\Aut(D,\dashv, \vdash) = \Aut(D,\vdash)$.
\end{corollary}

Dually, one can prove the following proposition.

\begin{proposition}\label{gisoRO} Let $(D_1,\dashv_1, \vdash_1)$ and $(D_2,\dashv_2, \vdash_2)$ be $g$-dimonoids such that  and $(D_1,\vdash_1)$ and $(D_2,\vdash_2)$ are right zero semigroups.  The $g$-dimonoids $(D_1,\dashv_1, \vdash_1)$ and $(D_2,\dashv_2, \vdash_2)$ are isomorphic if and only if the semigroups $(D_1,\dashv_1)$ and $(D_2,\dashv_2)$ are isomorphic.
\end{proposition}

\begin{proposition}\label{gisoabelian} Let $(D_1,\dashv_1, \vdash_1)$ and $(D_2,\dashv_2, \vdash_2)$ be abelian $g$-dimonoids.  The $g$-dimonoids $(D_1,\dashv_1, \vdash_1)$ and $(D_2,\dashv_2, \vdash_2)$ are isomorphic if and only if the semigroups $(D_1,\dashv_1)$ and  $(D_2,\dashv_2)$ are isomorphic.
\end{proposition}

\begin{proof}[Proof] It is immediate to show that if the $g$-dimonoids  $(D_1,\dashv_1, \vdash_1)$ and $(D_2,\dashv_2, \vdash_2)$ are isomorphic, then the semigroups $(D_1,\dashv_1)$ and $(D_2,\dashv_2)$ are isomorphic as well.
Conversely, let $\psi: D_1 \to D_2$ be an isomorphism from the semigroup $(D_1,\dashv_1)$ to the semigroup $(D_2,\dashv_2)$. Since $\psi(x\vdash_1 y) = \psi(y\dashv_1 x) = \psi(y)\dashv_2 \psi(x) = \psi(x)\vdash_2 \psi(y)$ for all $x,y\in D_1$, it follows that $\psi: D_1 \to D_2$ is an isomorphism from the $g$-dimonoid $(D_1,\dashv_1, \vdash_1)$ to the $g$-dimonoid $(D_2,\dashv_2, \vdash_2)$.
\end{proof}

\begin{corollary}\label{gaut_ab_dm} Let $(D,\dashv, \vdash)$ be an abelian $g$-dimonoid. Then $\Aut(D,\dashv, \vdash) = \Aut(D,\dashv) = \Aut(D,\vdash)$.
\end{corollary}

\section{Commutative rectangular $g$-dimonoids}\label{sec:com_rec_g-dm}

In this section, we describe, up to isomorphism, all commutative rectangular $g$-dimonoids and calculate their halos and automorphism groups. 

\medskip

The following proposition is an immediate consequence of the definitions.

\begin{proposition}\label{com_rec_sg} Each commutative rectangular semigroup $S$ is a null semigroup. 
\end{proposition}

We define a $g$-dimonoid $(D, \dashv, \vdash)$ to be  {\em iso-dual} if it is isomorphic to its dual $g$-dimonoid $(D, \dashv, \vdash)^d$. It is evident that every abelian (self-dual)  $g$-dimonoid is iso-dual. Theorem~\ref{iso-dual_O} provides an example of an iso-dual $g$-dimonoid that is not abelian.

\begin{theorem}\label{iso-dual_O}
Let $D$ be a set, and $0\in D$, $a\in D$. An algebraic structure $O_{D^0} \rbag O_{D^a} = (D, \dashv, \vdash)$, where $(D, \dashv)$ and $(D, \vdash)$ are null semigroups with zeros 
$0$ and $a$, respectively, is an iso-dual commutative rectangular $g$-dimonoid with $\Aut(O_{D^0} \rbag O_{D^a})\cong S_{D\setminus\{0,a\}}$. Moreover, if $0\ne a$, the $g$-dimonoid $O_{D^0} \rbag O_{D^a}$ is neither abelian nor a dimonoid, and if $|D|>1$, it has an empty halo. 

\end{theorem}

\begin{proof}[Proof] It follows directly from the definition of the $g$-dimonoid operations that, for all $x,y,z\in D$, 
$$(x \dashv y) \dashv z = 0 = x \dashv (y \vdash z)\ \text{ and }\ (x \dashv y) \vdash z = a =x \vdash (y \vdash z).$$ 
Therefore,  $(D, \dashv, \vdash)$ is a $g$-dimonoid. 

On the other hand, if $0\ne a$, then 
$$0 \dashv a=0\ne a = a\vdash 0\ \text{ and }\ (x \vdash y) \dashv z = 0\ne a= x \vdash (y \dashv z),$$ and hence the $g$-dimonoid $(D, \dashv, \vdash)$ is neither abelian nor a dimonoid.

\smallskip

Since both $(D, \dashv)$  and $(D, \vdash)$  are commutative (rectangular) semigroups, $(D, \dashv, \vdash)$ is a commutative (rectangular) $g$-dimonoid as well. 

\smallskip

Consider its dual $g$-dimonoid $O_{D^a} \rbag O_{D^0} = (D, \dashv^d, \vdash^d)$, where $(D, \dashv^d)$ is a null semigroup with zero $a$ and $(D, \vdash^d)$ is a null semigroup with zero $0$. Then the bijective map  $\psi:D\to D$ defined by $\psi(a)= 0$, $\psi(0)= a$, and $\psi(x)=x$ for all $D\setminus\{a,0\}$, is isomorphism from $(D, \dashv, \vdash)$ to $(D, \dashv^d, \vdash^d)$. Indeed, for all $x,y\in D$,
$$\psi(x \dashv y) = \psi(0) = a = \psi(x) \dashv^d \psi(y)\ \text{ and }\ \psi(x \vdash y)= \psi(a) = 0 = \psi(x) \vdash^d \psi(y).$$

\smallskip
 
Since for $|D|>1$ the  semigroup $(D, \dashv)$ contains no right identities, the halo of the $g$-dimonoid $O_{D^0} \rbag O_{D^a}$ is empty.

\smallskip

Let $\psi$ be an arbitrary automorphism of the $g$-dimonoid $(D,\dashv, \vdash)$. Then $\psi$ is an automor\-phism of the semigroup $(D, \dashv)$ and $\psi$ is an automorphism of the semigroup $(D, \vdash)$. Since automorphisms preserves zeros and $0$ is a zero $(D, \dashv)$, $a$ is a zero $(D, \vdash)$, it follows that $\psi(0)=0$ and $\psi(a)=a$.

On the other hand, let $f$ be any bijection of $D$ such that  $f(0) = 0$ and $f(a)= a$.  Then for all $x,y\in D$,
$$f(x \dashv y)= f(0) = 0 = f(x) \dashv f(y)\ \text{ and }\ f(x \vdash y)= f(a) = a = f(x) \vdash f(y).$$
It follows that any bijection of $D$ that preserves $0$ and $a$ generates an automorphism of the $g$-dimonoid $O_{D^0} \rbag O_{D^a}$. Therefore, $\Aut(O_{D^0} \rbag O_{D^a})\cong S_{D\setminus\{0,a\}}$.
\end{proof}

\begin{proposition}\label{iso_com_rec} Let $D$ be a set, and let $0\in D$. If $\{a,b\}\subset D\setminus\{0\}$, then the $g$-dimonoids  $O_{D^0} \rbag O_{D^a}$ and $O_{D^0} \rbag O_{D^b}$ are isomorphic while the $g$-dimonoids  $O_{D^0} \rbag O_{D^a}$ and $O_{D^0} \rbag O_{D^0} = O_{D^0}$ are not isomorphic, where for $p\in\{0,a,b\}$ the semigroup $O_{D^p}$ is a null semigroup with zero $p$.
\end{proposition}

\begin{proof}[Proof] One may verify that the dimonoids  $O_{D^0} \rbag O_{D^a}$ and $O_{D^0} \rbag O_{D^b}$ are isomorphic via the bijection $\psi: D \to D$ defined by $\psi(a) = b$, $\psi(b) = a$, and $\psi(x) = x$ for all $x\notin\{a,b\}$. 

To prove that the dimonoids  $O_{D^0} \rbag O_{D^0}$ and $O_{D^0} \rbag O_{D^a}$ are not isomorphic, suppose, to the contrary, that $\phi: O_{D^0} \rbag O_{D^0} \to O_{D^0} \rbag O_{D^a}$ is an isomorphism. Then $\phi$ is an automorphism of the semigroup $O_{D^0}$ and $\phi$ is an isomorphism from the semigroup $O_{D^0}$ to the semigroup $O_{D^a}$. Since automorphisms preserve zeros, and $0$ is zero of $O_{D^0}$,  we obtain $\phi(0) = 0$.  However, as isomorphisms preserve zeros and $0$ and $a$ are the zeros of $O_{D^0}$ and  $O_{D^a}$, respectively, we would have $\phi(0) = a$, leading to a contradiction.
\end{proof}

Taking into account that the  null semigroups $O_{S^0}$ and $O_{T^z}$ are isomorphic if and only if $|S|=|T|$, together with Propositions~\ref{com_rec_sg} and~\ref{iso_com_rec} and Theorem~\ref{iso-dual_O}, we obtain the following characterization theorem.

\begin{theorem}\label{char_iso_com_rec} Let $\kappa>1$ be a cardinal, and let $D$ be a set of cardinality $|D|=\kappa$ with distinct elements $0,z\in D$. Up to isomorphism, there exist exactly two commutative rectangular $g$-dimonoids of order $\kappa$: the trivial one $O_{D^0}$ with $\Aut(O_{D^0})\cong S_{D\setminus\{0\}}$, and the  nonabelian iso-dual nontrivial $g$-dimonoid $O_{D^0} \rbag O_{D^z}$ with $\Aut(O_{D^0} \rbag O_{D^z})\cong S_{D\setminus\{0,z\}}$, which is not a dimonoid.
\end{theorem}

\section{Some classes of noncommutative rectangular $g$-dimonoids}\label{sec:ncom_rec_g-dm}

This section is devoted to constructing noncommutative rectangular dimonoids.

\begin{proposition}\label{lrec_gdm}
Let  $(D,\dashv)$ be a left zero semigroup and $(D, \vdash)$ be an arbitrary semigroup. For an algebraic structure $(D,\dashv, \vdash)$ the following conditions are equivalent:
\begin{itemize}
\item[1)] $(D,\dashv, \vdash)$ is a $g$-dimonoid;
\item[2)] $(D,\dashv, \vdash)$ is a dimonoid;
\item[3)] $(D, \vdash)$ is a rectangular semigroup.
\end{itemize}
\end{proposition}
\begin{proof}[Proof] $(1)\Rightarrow(2)$ Since $(D,\dashv)$ is a left zero semigroup,  axiom $(D_2)$  holds: $$(x \vdash y) \dashv z = x \vdash y = x \vdash (y \dashv z),$$
and therefore $(D,\dashv, \vdash)$ is a dimonoid.
 
$(2)\Rightarrow(3)$ Given that $(D,\dashv)$ is a left zero semigroup, axiom $(D_3)$  yields
$$x \vdash (y \vdash z) = (x \dashv y) \vdash z =  x  \vdash z,$$
for all $x,y,z\in D$. Consequently, $(D, \vdash)$ is a rectangular semigroup.
  
$(3)\Rightarrow(1)$ Since $(D,\dashv)$ is a left zero semigroup, $$x \dashv (y \dashv z) = x = x \dashv (y \vdash z)$$ for all $x,y,z\in D$, so axiom $(D_1)$  holds.

Moreover, taking into account that  $(D, \vdash)$ is a rectangular semigroup, we conclude that $$(x \dashv y) \vdash z = x  \vdash z = x \vdash (y \vdash z)$$ for any $x,y,z\in D$, and hence  axiom $(D_3)$ is also satisfied. Consequently, $(D,\dashv, \vdash)$ is a $g$-dimonoid.
\end{proof}

Dually, one can prove the following proposition.

\begin{proposition}\label{rrec_gdm}
Let $(D, \dashv)$ be an arbitrary semigroup and $(D,\vdash)$ be a right zero semigroup. For an algebraic structure $(D,\dashv, \vdash)$ the following conditions are equivalent:
\begin{itemize}
\item[1)] $(D,\dashv, \vdash)$ is a $g$-dimonoid;
\item[2)] $(D,\dashv, \vdash)$ is a dimonoid;
\item[3)] $(D, \dashv)$ is a rectangular semigroup.
\end{itemize}
\end{proposition}

Since the semigroups $LO_{A_a\leftarrow S}$ and $RO_{A_a\leftarrow S}$ are rectangular, Propositions~\ref{lrec_gdm} and~\ref{rrec_gdm} imply that the rectangular $g$-dimonoids $LO_D \rbag LO_{A_a\leftarrow D}$, $LO_D \rbag RO_{A_a\leftarrow D}$, $LO_{A_a\leftarrow D} \rbag RO_D$, and $RO_{A_a\leftarrow D} \rbag RO_D$  are also rectangular dimonoids. These $g$-dimonoids were introduced and studied in~\cite{Gdim1}.

\smallskip

\begin{proposition}\label{lodim}
Let $(D, \dashv)$ be a semigroup  and $(D,\vdash)$ be a null semigroup with zero $0$. An algebraic structure $(D,\dashv, \vdash)$ is a $g$-dimonoid if and only if $x\dashv y\dashv z = x\dashv 0$ for all $x,y,z\in D$.
\end{proposition}
\begin{proof}[Proof]
Since $(D,\vdash)$ is a null semigroup,  the axiom $(D_3)$  holds:
$$(x \dashv y) \vdash z = 0 = x \vdash (y \vdash z).$$
  
Taking into account that  $x \dashv ( y \vdash z) = x \dashv 0$ for any $x,y,z\in D$, we conclude that  axiom $(D_1)$ holds if and only if $x \dashv y \dashv z = x \dashv 0$ for any $x,y,z\in D$.  
\end{proof}

The following theorem provides an example of noncommutative rectangular $g$-dimonoid which is not a dimonoid.

\begin{theorem} Let $A$ be a nonempty proper subset of a set $D$ such that $a\in A$ and $0\in D\setminus A$. Then $LO_{A_a\leftarrow D}\rbag  O_{D^0} = (D,\dashv, \vdash)$, where $(D,\vdash)$ is a null semigroup with zero $0$ and
\begin{center}
$x\dashv y=\begin{cases}
x\quad\text{if } x\in A \\
a\quad\text{if } x \in D\setminus A,
\end{cases}$
\end{center}
is a nonabelian rectangular $g$-dimonoid which is not a dimonoid, and $\Aut(LO_{A_a\leftarrow D}\rbag  O_{D^0}) \cong S_{A\setminus\{a\}}\times S_{D\setminus (A\cup \{0\})}$, $\Halo(LO_{A_a\leftarrow D}\rbag  O_{D^0}) = \emptyset$. Moreover, if $|A|>1$, then  $LO_{A_a\leftarrow D}\rbag  O_D$ is a  noncommutative $g$-dimonoid. 
\end{theorem}

\begin{proof}[Proof]
Using the definition of the semigroup $LO_{A_a\leftarrow D}$, it is immediate to check that $x\dashv y\dashv z = x\dashv 0$ for all $x,y,z\in D$. So, Proposition~\ref{lodim} implies that $LO_{A_a\leftarrow D}\rbag  O_{D^0}$ is a $g$-dimonoid. Taking into account that $O_{D^0}$ and $LO_{A_a\leftarrow D}$ are rectangular semigroups, we conclude that the $g$-dimonoid $LO_{A_a\leftarrow D}\rbag  O_{D^0}$ is rectangular as well. Since $0\vdash 0 = 0\ne a = 0\dashv 0$, it follows  that the $g$-dimonoid $LO_{A_a\leftarrow D}\rbag  O_{D^0}$ is not abelian. Furthermore, because $0 \vdash (0 \dashv 0) = 0\notin A$ and $(0 \vdash 0) \dashv 0 = 0 \dashv 0 = a\in A$, axiom $(D_2)$ fails to hold, and hence $(D,\dashv, \vdash)$ is not a dimonoid.

Let $\psi$ be an arbitrary automorphism of the dimonoid $(D,\dashv, \vdash)$. Then $\psi$ is an automorphism of the semigroup $(D, \dashv)$ and $\psi$ is an automorphism of the semigroup $(D, \vdash)$. Since automorphisms preserve zeros and $0$ is a zero of $(D, \vdash)$, we conclude that $\psi(0) = 0$.
Taking into account that automorphisms preserves left zeros and $A$ is a left zero subsemigroup of $(D, \dashv)$, we conclude that $\psi(A)=A$, and thus $\psi(D\setminus A)=D\setminus A$. Because $0\notin A$, we obtain  $\psi(a)=\psi(0\dashv a) = \psi(0)\dashv \psi(a)= 0\dashv \psi(a) = a$.

On the other hand, let $f$ be any bijection of $D$ such that $f(A)=A$, $f(a)=a$, and $f(0)=0$.  If $x\in A$, then $f(x)\in A$, and hence $f(x\dashv y)= f(x)= f(x)\dashv f(y)$ for all $y\in D$. In the case $x\in D\setminus A$ we have that $f(x)\in D\setminus A$, and thus $f(x\dashv y)= f(a)= a = f(x)\dashv f(y)$ for all $y\in D$. Moreover, $f(x\vdash y)= f(0)= 0 = f(x)\vdash f(y)$ for all $x, y\in D$. It follows that any bijection of $A$ that preserves $a$  and any bijection of $D\setminus A$ that preserves $0$ generate an automorphism of $LO_{A_a\leftarrow D}\rbag  O_{D^0}$. Therefore, $\Aut(LO_{A_a\leftarrow D}\rbag  O_{D^0})\cong S_{A\setminus\{a\}}\times S_{D\setminus (A\cup \{0\})}$.

For $|D|> 1$, the commutative semigroup $O_{D^0}$  has no identity, and therefore  $\Halo(LO_{A_a\leftarrow D}\rbag  O_{D^0}) = \emptyset$.

Since in the case $|A|>1$ the semigroup $LO_{A_a\leftarrow D}$ contains at least two left zeros, and hence it is not commutative, we conclude that the $g$-dimonoid $LO_{A_a\leftarrow D}\rbag  O_{D^0}$ is not commutative as well.   
\end{proof}

Dually, one can prove the following proposition and theorem.

\begin{proposition}\label{null_dm}
Let $(D, \dashv)$ be a null semigroup with zero $0$ and $(D, \vdash)$ be an arbitrary semigroup. An algebraic structure $(D,\dashv, \vdash)$ is a $g$-dimonoid if and only if $x\vdash y\vdash z = 0 \vdash z$ for all $x,y,z\in D$.
\end{proposition}

\begin{theorem} Let $A$ be a nonempty proper subset of a set $D$ such that $a\in A$ and $0\in D\setminus A$. Then $O_{D^0} \rbag RO_{A_a\leftarrow D}   = (D,\dashv, \vdash)$, where $(D,\dashv)$ is a null semigroup with zero $0$ and
\begin{center}
$x\vdash y=\begin{cases}
y\quad\text{if } y\in A \\
a\quad\text{if } y \in D\setminus A,
\end{cases}$
\end{center}
is a nonabelian rectangular $g$-dimonoid which is not a dimonoid, and $\Aut(O_{D^0} \rbag RO_{A_a\leftarrow D})  \cong S_{A\setminus\{a\}}\times S_{D\setminus (A\cup \{0\})}$, $\Halo(O_{D^0} \rbag RO_{A_a\leftarrow D}) = \emptyset$. Moreover, if $|A|>1$, then  $O_{D^0} \rbag RO_{A_a\leftarrow D}$ is a  noncommutative $g$-dimonoid. 
\end{theorem}

\section{Two-element $g$-dimonoids and their automorphism groups}\label{2eldim}

In this section we describe, up to isomorphism, all two-element $g$-dimonoids and their automorphism groups.

It is well-known that there are exactly five pairwise nonisomorphic semigroups having two elements: the multiplicative cyclic group $C_2=\{-1,1\}$, the linear semilattice $L_2=\{0,1\}$ with $\min$-operation, the null semigroup $O_2=\{0,1\}$ with zero $0$, the left zero semigroup $LO_2$ with operation $xy=x$, and the right zero semigroup $RO_2$ with operation $xy=y$.

\begin{theorem}
Up to isomorphism, there exist $9$ two-element $g$-dimonoids, among which $4$ dimonoids are abelian. Moreover, up to isomorphism, there are $4$ commutative $g$-dimonoids of order $2$: three abelian trivial $g$-dimonoids and one nonabelian nontrivial $g$-dimonoid. Furthermore, there are $7$ rectangular $g$-dimonoids of order $2$, including two commutative iso-dual $g$-dimonoids, two pairs of dual noncommutative nonabelian $g$-dimonoids and a single noncommutative abelian $g$-dimonoid. Additionally, there exist exactly $5$ pairwise nonisomorphic two-element trivial $g$-dimonoids.
\end{theorem}

\begin{proof}[Proof]
In the sequel, we divide our investigation into cases. In the case of a semigroup $(\{a,b\},*)$, we shall find all pairwise nonisomorphic dimonoids $(D,\dashv, \vdash)$ such that $(D,\dashv)$ is isomorphic to $(\{a,b\},*)$.

\smallskip
{\noindent \bf Cases $C_2$ and $L_2$}.  According to Proposition~\ref{grl_id}, if a semigroup $(D,\dashv)$ possesses a left identity or a semigroup $(D,\vdash)$ possesses a right identity, then the operations of a $g$-dimonoid $(D,\dashv, \vdash)$ coincide. Therefore, up to isomorphism, there exist a unique $g$-dimonoid $(D,\dashv, \vdash)$ such that $(D,\dashv)\cong C_2$ or $(D,\vdash)\cong C_2$, and this $g$-dimonoid is the trivial dimonoid $C_2$. Similarly, $L_2$ is a unique $g$-dimonoid in the class of $g$-dimonoids $(D,\dashv, \vdash)$ such that $(D,\dashv)\cong L_2$ or $(D,\vdash)\cong L_2$. The trivial $g$-dimonoids $C_2$ and $L_2$ are commutative and abelian.

\smallskip

Note, that in all the remaining cases, the semigroups $LO_2$, $RO_2$, and $O_2$ are rectangular.

\smallskip

{\noindent \bf Case $LO_2$}. If $(D,\vdash)\cong RO_2$, then by Proposition~\ref{lrec_gdm} we obtain the noncommutative abelian rectangular $g$-dimonoid $LO_2 \rbag RO_2$.
According to Proposition~\ref{gisoabelian}, $LO_2 \rbag RO_2$ is the unique $g$-dimonoid in the class of abelian $g$-dimonoids $(D,\dashv,\vdash)$ satisfying $(D,\dashv)\cong LO_2$ and  $(D,\vdash)\cong RO_2$.
It follows from Corollary~\ref{gaut_ab_dm} that $\Aut(LO_2 \rbag RO_2)=\Aut(LO_2)=S_2\cong C_2$.

\smallskip
In the case $(D,\vdash)\cong O_2$, Proposition~\ref{lrec_gdm} yields the noncommutative nonabelian rectangu\-lar $g$-dimonoid $LO_2 \rbag O_2$.
By Proposition~\ref{gisoLO}, $LO_2 \rbag O_2$ is the unique $g$-dimonoid in the class of dimonoids $(D,\dashv,\vdash)$ such that $(D,\dashv)\cong LO_2$ and $(D,\vdash)\cong O_2$.
Consequently, by Corollary~\ref{gaut_LO_dm}, $\Aut(LO_2 \rbag O_2)=\Aut(O_2)\cong C_1$.

\smallskip
In the remaining case, we obtain the  noncommutative nonabelian rectangular trivial $g$-dimonoid $LO_2$.
By Proposition~\ref{gisoLO}, $LO_2$ is the unique $g$-dimonoid in the class of $g$-dimonoids $(D,\dashv,\vdash)$ satisfying $(D,\dashv)\cong (D,\vdash)\cong LO_2$.

\smallskip
{\noindent \bf Case $RO_2$}. Each element of the semigroup $RO_2$ is a left identity.
According to Proposition~\ref{grl_id}, if a semigroup $(D,\dashv)$ has a left identity, then the two operations of a $g$-dimonoid $(D,\dashv,\vdash)$ coincide.
Hence, by Proposition~\ref{gisoRO}, the trivial noncommutative nonabelian rectangular $g$-dimonoid $RO_2$ is the unique $g$-dimonoid in the class of $g$-dimonoids $(D,\dashv,\vdash)$ with $(D,\dashv)\cong RO_2$.
This dimonoid is dual to $LO_2$.

\smallskip
{\noindent \bf Case $O_2$}.  
If $(D,\vdash)$ is a right zero semigroup, then Proposition~\ref{rrec_gdm} yields  the noncommutative nonabelian rectangular $g$-dimonoid $O_2\rbag RO_2$, which is unique in the class of $g$-dimonoids $(D,\dashv, \vdash)$ such that $(D,\dashv)\cong O_2$ and $(D,\vdash)\cong RO_2$, in accordance with Proposition~\ref{gisoRO}. This $g$-dimonoid is dual to the $g$-dimonoid $LO_2\rbag O_2$, and hence by Corollary~\ref{aut_gdm_dual}, $\Aut(O_2\rbag RO_2) = \Aut(LO_2\rbag O_2) \cong C_1$.

According to Proposition~\ref{grl_id}, if $(D,\vdash)$ contains a right identity, then the operations of a $g$-dimonoid $(D,\dashv, \vdash)$ coincide. Since each element of the semigroup $LO_2$ is a right identity, there does not exist a $g$-dimonoid $(D,\dashv, \vdash)$ such that $(D,\dashv) \cong O_2$ and $(D,\vdash)\cong LO_2$.

In the final case, we obtain by Theorem~\ref{char_iso_com_rec} the abelian commutative  regular trivial $g$-dimonoid $O_2$ and the nonabelian commutative regular iso-dual nontrivial $g$-dimonoid  $O_{D^a} \rbag O_{D^b}$ with $\Aut(O_{D^a} \rbag O_{D^b})\cong C_1$.
\end{proof}

Table~\ref{tab:g-dm2} lists all pairwise nonisomorphic two-element $g$-dimonoids and their corresponding automorphism groups.

\begin{table}[ht]
    \centering
\resizebox{14cm}{!}{
\begin{tabular}{|c|c|c|c|c|c|c|c|c|c|}
        \hline
        $D$ & $C_2$  & $L_2$ & $O_2$  & $LO_2$ & $RO_2$ & $LO_2 \rbag RO_2$ & $LO_2 \rbag O_2$ & $O_2 \rbag RO_2$ & $O_{D^a} \rbag O_{D^b}$ \\
        \hline
        $\Aut(D)$ &  $C_1$ & $C_1$ & $C_1$ & $C_2$ & $C_2$ & $C_2$  & $C_1$ & $C_1$ & $C_1$  \\
        \hline
\end{tabular}
}
\smallskip
\caption{Nonisomorphic $2$-element $g$-dimonoids and their automorphism groups}\label{tab:g-dm2}
\end{table}

\section{Number of $g$-dimonoids of small order}\label{sec:gdm-finite}

The determination of the numbers $\mathrm{s}(n)$ and $\mathrm{cs}(n)$, representing respectively the counts of all pairwise nonisomorphic semigroups of order 
$n$ and all pairwise nonisomorphic commutative semigroups of order $n$, is a difficult combinatorial problem. The functions $\mathrm{s}(n)$ and $\mathrm{cs}(n)$
grow very rapidly as $n$ tends to infinity. The sequences $(\mathrm{s}(n))$ and $(\mathrm{cs}(n))$ are listed in the On-Line Encyclopedia of Integer Sequences as entries A027851 and A001426, respectively. All currently known exact values of these sequences are presented in Tables~\ref{tab:sg} and~\ref{tab:csg}.

\medskip

\begin{table}[H]
    \centering
    \resizebox{14cm}{!}{
        \begin{tabular}{|r|c|c|c|c|c|c|c|c|c|c|c|}
            \hline
            $n$\ \ &  1  & 2 & 3 &   4 & 5  & 6 & 7 & 8 & 9  \\
            \cline{1-10}
            $\mathrm{s}(n)$  & 1 & 5 & 24 & 188 & 1915 & 28634 & 1627672 & 3684030417 & 105978177936292  \\
            \cline{1-10}
         \end{tabular}
    }
    \smallskip
    \caption{Number of nonisomorphic semigroups up to order $9$}\label{tab:sg}
\end{table}

\begin{table}[H]
    \centering
    \resizebox{13.5cm}{!}{
        \begin{tabular}{|r|c|c|c|c|c|c|c|c|c|c|c|}
            \hline
            $n$\ \  & 1  & 2 & 3 &   4 & 5  & 6 & 7 & 8 & 9 & 10 \\
            \cline{1-11}
            $\mathrm{cs}(n)$  & 1 & 3 & 12 & 58 & 325 & 2143 & 17291 & 221805 & 11545843 & 3518930337 \\
            \cline{1-11}
         \end{tabular}
    }
    \smallskip
    \caption{Number of nonisomorphic commutative semigroups up to order $10$}\label{tab:csg}
\end{table}

Denote by $\mathrm{gdm}(n)$, $\mathrm{cgdm}(n)$, $\mathrm{agdm}(n)$, and $\mathrm{rgdm}(n)$ the number of all pairwise nonisomorphic $g$-dimonoids, commutative $g$-dimonoids, abelian $g$-dimonoids, and rectangular $g$-dimonoids of order $n$, respectively. We were able to calculate these cardinalities for small $n$. In Appendix~\ref{appnd}, we explain the method used to generate all pairwise nonisomorphic $g$-dimonoids of order $n$ and provide a listing of the \texttt{Python} code employed for these computations. The results of (computer) calculations are presented in Tables~\ref{tab:gdm}--\ref{tab:rgdm}.

\begin{table}[H]
    \centering
    \resizebox{7cm}{!}{
        \begin{tabular}{|r|c|c|c|c|c|c|}
            \hline
            $n$\ \ &  1  & 2 & 3 &   4 & 5   \\
            \cline{1-6}
            $\mathrm{gdm}(n)$  & 1 & 9 & 78 & 1693 & 162465  \\
            \cline{1-6}
         \end{tabular}
    }
    \smallskip
    \caption{Number of  nonisomorphic $g$-dimonoids up to order 5}\label{tab:gdm}
\end{table}

\begin{table}[H]
    \centering
    \resizebox{8cm}{!}{
        \begin{tabular}{|r|c|c|c|c|c|c|c|}
            \hline
            $n$\ \ &  1  & 2 & 3 &   4 & 5 & 6  \\
            \cline{1-7}
            $\mathrm{cgdm}(n)$  & 1 & 4 & 22 & 189 & 3249 &  256253 \\
            \cline{1-7}
        \end{tabular}
    }
    \smallskip
    \caption{ Number of  nonisomorphic commutative $g$-dimonoids up to order 6}\label{tab:cgdm}
\end{table}

Since a $g$-dimonoid is considered trivial when its two operations coincide, the numbers of all pairwise nonisomorphic nontrivial $g$-dimonoids and  pairwise nonisomorphic  commutative nontrivial $g$-dimonoids of order $n$ are equal to $\mathrm{gdm}(n)-\mathrm{s}(n)$ and  $\mathrm{cgdm}(n)-\mathrm{cs}(n)$, respectively.

\begin{table}[H]
    \centering
    \resizebox{7.5cm}{!}{
        \begin{tabular}{|r|c|c|c|c|c|c|c|}
            \hline
            $n$\ \ &  1  & 2 & 3 &   4 & 5 & 6  \\
            \cline{1-7}
            $\mathrm{agdm}(n)$ &  1 & 4 & 17 & 103 & 791 & 10870 \\
            \cline{1-7}
        \end{tabular}
    }
    \smallskip
    \caption{ Number of  nonisomorphic abelian $g$-dimonoids up to order 6}\label{tab:agdm}
\end{table}

Since a commutative $g$-dimonoid is abelian if and only if it is trivial, the number of all pairwise nonisomorphic trivial abelian $g$-dimonoids of order $n$ is equal to $\mathrm{cs}(n)$, and hence the numbers of all pairwise nonisomorphic nonabelian commutative $g$-dimonoids and pairwise nonisomorphic noncommutative abelian dimonoids of order~$n$ are $\mathrm{cgdm}(n)\!-\!\mathrm{cs}(n)$ and $\mathrm{agdm}(n)\!-\!\mathrm{cs}(n)$, respectively.

\begin{table}[H]
    \centering
    \resizebox{7.5cm}{!}{
        \begin{tabular}{|r|c|c|c|c|c|c|}
            \hline
            $n$\ \ &  1  & 2 & 3 &   4 & 5 & 6  \\
            \cline{1-7}
            $\mathrm{rgdm}(n)$  & 1 & 7 & 27 & 128 & 555 & 2846 \\
            \cline{1-7}
            $\mathrm{rs}(n)$  & 1 & 3 & 5 & 10 & 14 & 27 \\
            \cline{1-7}
        \end{tabular}
    }
    \smallskip
    \caption{ Number of  nonisomorphic rectangular $g$-dimonoids and semigroups up to order 6}\label{tab:rgdm}
\end{table}

Table~\ref{tab:rgdm} also lists the numbers $\mathrm{rs}(n)$ of all pairwise nonisomorphic rectangular semigroups for $n \leq 6$. Hence, the number of pairwise nonisomorphic rectangular nontrivial $g$-dimonoids of order~$n$ is given by $\mathrm{rgdm}(n)\! -\! \mathrm{rs}(n)$.

According to Theorem~\ref{char_iso_com_rec}, the numbers of pairwise nonisomorphic noncommutative rectan\-gular $g$-dimonoids and pairwise nonisomorphic noncommutative rectangular nontrivial $g$-dimonoids of order~$n>1$ are equal to $\mathrm{rgdm}(n)\! - \! 2$ and $\mathrm{rgdm}(n)\! -\! \mathrm{rs}(n)\!-\! 1$, respectively.

\begin{problem}Determine the numbers $\mathrm{gdm}(n)$ for $n\geq 6$ and $\mathrm{cgdm}(n)$, $\mathrm{agdm}(n)$, $\mathrm{rgdm}(n)$ for $n\geq 7$.
\end{problem}

\appendix

\section{Program code for computing nonisomorphic $g$-dimonoids}\label{appnd}

This code processes the Cayley tables of all nonisomorphic semigroups of order $n$ obtained with \texttt{GAP} using the \texttt{Smallsemi} library of semigroups of small order.  The corresponding  tables were relabeled from $\{1,\ldots,n\}$ to $\{0,\ldots,n\!-\!1\}$ for computational convenience and exported to \texttt{csv}-files. The code loads the Cayley tables from \texttt{csv}-files, generates all their permutations, and checks pairwise combinations for compliance with the system of $g$-dimonoid axioms. It then eliminates isomorphic cases, retaining only representatives of isomorphism classes, and produces a complete list of pairwise nonisomorphic $g$-dimonoids of a given order $n$. The results are saved as operation tables for $\dashv$ and $\vdash$ into a single \texttt{csv}-file for further analysis. For $n=6$, we translated this code from \texttt{Python} into \texttt{C++} and performed parallel computations. We present the \texttt{Python} version here, as it is more readable and accessible to the reader.

\begin{lstlisting}
import csv
import glob
import os
from itertools import permutations, product

# ----------------------------
# 1) Load semigroup tables
# ----------------------------
semigroup_tables = []
n = None  # size will be determined automatically

for filename in glob.glob("semigroup_*.csv"):
    with open(filename, newline='') as csvfile:
        reader = csv.reader(csvfile)
        table = []
        for row in reader:
            nums = [int(cell) for cell in row if cell.strip() != ""]
            if n is None:
                n = len(nums)  # determine size from the first row
            table.append(nums)
        semigroup_tables.append(table)

print(f"Loaded {len(semigroup_tables)} semigroup tables (order {n})")

# ----------------------------
# 2) Generate all permutations for each table
# ----------------------------
all_tables = []
for tab in semigroup_tables:
    for p in permutations(range(n)):
        inv = [0]*n
        for idx, val in enumerate(p):
            inv[val] = idx
        newtab = [[p[tab[inv[i]][inv[j]]] for j in range(n)] for i in range(n)]
        all_tables.append(newtab)

print(f"Total associative tables with permutations: {len(all_tables)}")

# ----------------------------
# 3) Check g-dimonoid axioms
# ----------------------------
def is_dimonoid(op_dashv, op_vdash):
    for x, y, z in product(range(n), repeat=3):
        if op_dashv[op_dashv[x][y]][z] != op_dashv[x][op_vdash[y][z]]: return False #D1
        if op_vdash[op_dashv[x][y]][z] != op_vdash[x][op_vdash[y][z]]: return False #D3
        #if op_dashv[x][y] != op_dashv[y][x]: return False #CommL
        #if op_vdash[x][y] != op_vdash[y][x]: return False #CommR
        #if op_vdash[x][y] != op_dashv[y][x]: return False #Abelian
        #if op_dashv[x][op_dashv[y][z]] != op_dashv[x][z] or op_vdash[x][op_vdash[y][z]] != op_vdash[x][z]: return False #Rectangular
    return True

# ----------------------------
# 4) Find all g-dimonoids
# ----------------------------
dimonoid_pairs = []
for t1 in all_tables:
    for t2 in all_tables:
        if is_dimonoid(t1, t2):
            dimonoid_pairs.append((t1, t2))

print(f"Found {len(dimonoid_pairs)} ordered dimonoid pairs")

# ----------------------------
# 5) Canonization under permutations
# ----------------------------
canon_seen = set()
rep_pairs = []  # list of representatives to output
for t1, t2 in dimonoid_pairs:
    keys = []
    for p in permutations(range(n)):
        inv = [0]*n
        for idx, val in enumerate(p):
            inv[val] = idx
        r1 = tuple(tuple(p[t1[inv[i]][inv[j]]] for j in range(n)) for i in range(n))
        r2 = tuple(tuple(p[t2[inv[i]][inv[j]]] for j in range(n)) for i in range(n))
        keys.append((r1, r2))
    min_key = min(keys)
    if min_key not in canon_seen:
        canon_seen.add(min_key)
        rep_pairs.append(min_key)

print(f"Number of nonisomorphic g-dimonoids (order {n}): {len(canon_seen)}")

# ----------------------------
# 6) Save all nonisomorphic g-dimonoids into one csv-file
# ----------------------------
out_dir = os.path.dirname(os.path.abspath(file))  # program directory
list_file = os.path.join(out_dir, f"list_gdm{n}.csv")

with open(list_file, "w", newline='', encoding="utf-8") as f:
    writer = csv.writer(f)
    for idx, (r1, r2) in enumerate(rep_pairs, 1):
    		writer.writerow([f"g-dimonoid #{idx}"])
    		for name, table in [("Operation dashv", r1), ("Operation vdash", r2)]:
        		writer.writerow([name])
        		writer.writerows(table)
        		writer.writerow([])
    		writer.writerow([])
     
print(f"\nAll nonisomorphic g-dimonoids saved to {list_file}")
\end{lstlisting}

\end{document}